\documentclass[twoside,11pt,reqno]{amsart}
\usepackage{eurosym}
\usepackage{amsmath,amsthm,amssymb,amstext,amsfonts,amscd}
\usepackage{graphicx}
\usepackage{multirow}

\newtheorem{theorem}{Theorem}

\newtheorem{corollary}{Corollary}
\newtheorem{definition}{Definition}

\newtheorem{lemma}{Lemma}

\newtheorem{remark}{Remark}

\newtheorem{proposition}{Proposition}
\numberwithin{equation}{section}

\begin{document}
\title[Growth of $(\alpha ,\beta ,\gamma )$-order solutions of linear
differential equations]{Growth of $(\alpha ,\beta ,\gamma )$-order solutions
of linear differential equations with analytic coefficients in the unit disc}
\author[A. H. Arrouche and B. Bela\"{\i}di]{Amina Halima Arrouche and
Benharrat Bela\"{\i}di$^{\star}$}
\address{A. H. Arrouche: Department of Mathematics, Laboratory of Pure and
Applied Mathematics, University of Mostaganem (UMAB), B. P. 227
Mostaganem-(Algeria).}
\email{aminahalima.arrouche@etu.univ-mosta.dz}
\address{B. Bela\"{\i}di $^{\star}$ Corresponding author: Department of
Mathematics, Laboratory of Pure and Applied Mathematics, University of
Mostaganem (UMAB), B. P. 227 Mostaganem-(Algeria).}
\email{benharrat.belaidi@univ-mosta.dz}

\begin{abstract}
In this paper, we study the growth of solutions to higher-order complex
linear differential equations in the unit disc, where the analytic
coefficients are of finite $(\alpha ,\beta ,\gamma )$-order. By employing
the concepts of $(\alpha ,\beta ,\gamma )$-order and $(\alpha ,\beta ,\gamma
)$-type, we establish new results concerning the growth of such solutions.
These results extend and generalize previous work by the second author and
by Biswas.
\end{abstract}

\maketitle

\section{Introduction and Definitions}

\noindent Throughout this paper, we assume that the reader is familiar with
the basic concepts, notation, and fundamental results of Nevanlinna theory
in the complex plane and in the unit disc $\Delta =\left\{ z\in \mathbb{C}%
:\left\vert z\right\vert <1\right\} ,$ see for example, \cite%
{17,19,20,24,25,31,33}.

\noindent \qquad For $k\geq 2,$ we consider the complex linear differential
equation
\begin{equation}
f^{(k)}+A_{k-1}(z)f^{(k-1)}+\cdots +A_{0}(z)f=0,  \label{1.1}
\end{equation}%
where the coefficients $A_{j}$ ($j=0,1,\dots ,k-1$) are analytic in the unit
disc $\Delta .$ It is well known that every solution of (\ref{1.1}) is
analytic in $\Delta $, and that the equation possesses exactly $k$ linearly
independent solutions (see e.g. \cite{20}). The study of growth and
oscillation of solutions of complex linear differential equations in the
unit disc has developed rapidly since the 1980s, see \cite{28}. A systematic
investigation in this direction was initiated by Heittokangas \cite{20}, who
introduced suitable function spaces to describe the growth of solutions when
the coefficients $A_{j}$ ($j=0,1,\dots ,k-1$) of (\ref{1.1}) are analytic
functions in $\Delta $.

Subsequently, Heittokangas et al. \cite{22} investigated the iterated $p$%
-order of solutions of equation (\ref{1.1}), for $k\geq 2,$ assuming that
the coefficient $A_{0}$ is dominant. Their result shows that the growth of
every non-trivial solution is completely determined by the iterated $p$%
-order of $A_{0}$.

\begin{theorem}
\label{teo1.1}(\cite{22}) Let $p\in \mathbb{N}$. If the coefficients $%
A_{0},A_{1},\ldots ,A_{k-1}$ are analytic functions in $\Delta $ such that $%
\rho _{M,p}\left( A_{j}\right) <\rho _{M,p}\left( A_{0}\right) $ for $%
j=1,...,k-1,$ then every solution $f\not\equiv 0$ of $(\ref{1.1})$ satisfies
$\rho _{M,p+1}\left( f\right) =\rho _{M,p}\left( A_{0}\right) .$
\end{theorem}

Observe that $A_{0}(z)$ is the only one dominant coefficient. Later, Hamouda
\cite{18}, extended Theorem \ref{teo1.1} by allowing more than one dominant
coefficient, replacing strict dominance by suitable conditions involving
iterated $p$-order and iterated $p$-type. He proved the following theorem.

\begin{theorem}
\label{teo1.2}(\cite{18}) Let $p\in \mathbb{N}$. If the coefficients $%
A_{0},A_{1},\ldots ,A_{k-1}$ are analytic functions in $\Delta $ such that%
\begin{equation*}
\rho _{M,p}\left( A_{j}\right) \leq \rho _{M,p}\left( A_{0}\right) <\infty ,%
\text{ }j=1,...,k-1,
\end{equation*}%
and
\begin{equation*}
\max \left\{ \tau _{M,p}\left( A_{j}\right) :\rho _{M,p}\left( A_{j}\right)
=\rho _{M,p}\left( A_{0}\right) >0\right\} <\tau _{M,p}\left( A_{0}\right) ,
\end{equation*}%
then every solution $f\not\equiv 0$ of $(\ref{1.1})$ satisfies $\rho
_{M,p+1}\left( f\right) =\rho _{M,p}\left( A_{0}\right) .$
\end{theorem}

Various further extensions of Theorem \ref{teo1.1} and Theorem \ref{teo1.2}
were obtained using the notion of $[p,q]$-order, see for example, \cite%
{2,3,4,26,32}. However, these growth indicators are not sufficient to
describe arbitrary growth behaviour. In fact, it was shown in \cite[Example
1.4]{16}, that for every $p\in
%TCIMACRO{\U{2115} }%
%BeginExpansion
\mathbb{N}
%EndExpansion
$ there exist functions whose iterated $p$-order and $[p,q]$-order are both
infinite. To overcome this limitation, Chyzhykov and Semochko \cite{16}
introduced the concept of $\varphi $-order, which provides a more flexible
scale for measuring growth, see also \cite{29}. Using this notion, Semochko
\cite{29} obtained a result that improves the above-mentioned theorem of
Heittokangas et al. by relaxing the dominance condition on the coefficients.

Let $\varphi $ be an increasing unbounded function on $\left( 0,+\infty
\right) .$ The $\varphi $-orders of an analytic function $f$ in $\Delta $
are defined by (\cite{29})
\begin{equation*}
\tilde{\rho}_{\varphi }^{0}(f)=\underset{r\longrightarrow 1^{-}}{\lim \sup }%
\dfrac{\varphi (M(r,f))}{-\log (1-r)},\qquad \tilde{\rho}_{\varphi }^{1}(f)=%
\underset{r\longrightarrow 1^{-}}{\lim \sup }\dfrac{\varphi (\log M(r,f))}{%
-\log (1-r)}.
\end{equation*}%
If $f$ is meromorphic in $\Delta ,$ then the $\varphi $-orders are defined
by
\begin{equation*}
\rho _{\varphi }^{0}(f)=\underset{r\longrightarrow 1^{-}}{\lim \sup }\dfrac{%
\varphi (e^{T(r,f)})}{-\log (1-r)},\qquad \rho _{\varphi }^{1}(f)=\underset{%
r\longrightarrow 1^{-}}{\lim \sup }\dfrac{\varphi (T(r,f))}{-\log (1-r)}.
\end{equation*}

Note that for $\varphi (r)=\log _{p+1}^{+}r,\ p\in \mathbb{N}$ and if $f$ is
an analytic function in $\Delta ,$ then
\begin{equation*}
\tilde{\rho}_{\varphi }^{0}(f)=\rho _{M,p}(f)\text{ and }\tilde{\rho}%
_{\varphi }^{1}(f)=\rho _{M,p+1}(f).
\end{equation*}%
The following theorem due to Semochko \cite{29} used the concept of $\varphi
$-order which improves Theorem \ref{teo1.1}.

\begin{theorem}
\label{teo1.3} (\cite{29}) Let $\varphi \in \Phi $ and $A_{0},A_{1},\ldots
,A_{k-1}$ be analytic functions in $\Delta $ such that $\max \{\tilde{\rho}%
_{\varphi }^{0}(A_{j}),j=1,\ldots ,k-1\}<\tilde{\rho}_{\varphi }^{0}(A_{0}).$
Then, every solution $f\not\equiv 0$ of $(\ref{1.1})$ satisfies $\tilde{\rho}%
_{\varphi }^{1}(f)=\tilde{\rho}_{\varphi }^{0}(A_{0}).$
\end{theorem}

The generalized $(\alpha ,\beta )$-order of an entire function was
introduced by Sheremeta \cite{30}, and has been studied extensively in
recent years; see, for instance, \cite{10,11}. Applications of this concept
to differential equations were initiated by Mulyava et al. \cite{27} applied
the $(\alpha ,\beta )$-order to the study of solutions of certain
second-order heterogeneous differential equations, obtaining several
remarkable results. For details on the $(\alpha ,\beta )$-order, we refer to
\cite{5,7,27,30}.

Motivated by these developments, we study equation (\ref{1.1}) within a more
general growth framework. The main purpose of the present paper is to
investigate the growth of non-trivial solutions of equation (\ref{1.1}) in
the unit disc in terms of their $(\alpha ,\beta ,\gamma )$-order and $%
(\alpha ,\beta ,\gamma )$-type. Our results extend and unify several earlier
theorems concerning iterated $p$-order, $[p,q]$-order, and $\varphi $-order,
and provide a more general description of the growth behaviour of solutions
in the unit disc.

Let $L$ be the class of continuous, non-negative functions $\alpha :(-\infty
,+\infty )\rightarrow \left[ 0,+\infty \right) $ such that $\alpha
(x)=\alpha (x_{0})\geq 0$ for $x\leq x_{0}$ and $\alpha (x)\uparrow +\infty $
as $x_{0}\leq x\rightarrow +\infty $. We say that $\alpha \in L_{1}$, if $%
\alpha \in L$ and $\alpha (a+b)\leq \alpha (a)+\alpha (b)+c$ for all $%
a,b\geq R_{0}$ and fixed $c\in (0,+\infty )$. Further, we say that $\alpha
\in L_{2}$, if $\alpha \in L$ and $\alpha (x+O(1))=(1+o(1))\alpha (x)$ as $%
x\rightarrow +\infty $. Finally, $\alpha \in L_{3}$, if $\alpha \in L$ and $%
\alpha $ is subadditive, that is, $\alpha (a+b)\leq \alpha (a)+\alpha (b)$
for all $a,b\geq R_{0}$. Clearly $L_{3}\subset L_{1}$.

Particularly, if $\alpha \in L_{3}$, then for any integer $m\geq 2$, $\alpha
(mr)\leq m\alpha (r)$.

Concavity also implies subadditivity: if $\alpha (r)$ is concave on $%
[0,+\infty )$ with $\alpha (0)\geq 0$, then $\alpha (tx)\geq t\alpha (x)$,
for $t\in \lbrack 0,1]$, which yields%
\begin{equation*}
\alpha (a+b)\leq \alpha (a)+\alpha (b)
\end{equation*}%
for $a,b\geq 0$. Moreover, if $\alpha $ is non-decreasing, subadditive, and
unbounded, then for any $R_{0}\geq 0$%
\begin{equation*}
\alpha (r)\leq \alpha (r+R_{0})\leq \alpha (r)+\alpha (R_{0})
\end{equation*}%
hence $\alpha (r)\sim \alpha (r+R_{0})$ as $r\rightarrow +\infty $.

We assume throughout the paper that $\alpha ,$ $\beta $ and $\gamma $
satisfy the following conditions : (i) Always $\alpha \in L_{1},$ $\beta \in
L_{2}$ and $\gamma \in L_{3}$; and (ii) $\alpha (\log ^{[p]}x)=o(\beta (\log
\gamma (x))),$ $p\geq 2,$ $\alpha (\log x)=o(\alpha \left( x\right) )$ and $%
\alpha ^{-1}(kx)=o\left( \alpha ^{-1}(x)\right) $ $\left( 0\leq k<1\right) $
as $x\rightarrow +\infty $.

Unless otherwise stated, we assume these conditions hold.

Recently, Heittokangas et al. \cite{23} introduced the concept of $\varphi $%
-order of entire and meromorphic functions, where $\varphi $ is a
subadditive function. Building on this idea, the second author and Biswas
\cite{6} introduced the $(\alpha ,\beta ,\gamma )$-order of a meromorphic
function in the complex plane.

Several works concerning the growth of solutions of higher-order
differential equations in terms of $(\alpha ,\beta ,\gamma )$-order have
since appeared; see \cite{6}, \cite{8} and \cite{9}.

For $x\in \lbrack 0,+\infty )$ and $k\in \mathbb{N}$ where $%
%TCIMACRO{\U{2115} }%
%BeginExpansion
\mathbb{N}
%EndExpansion
$ is the set of all positive integers, we define the iterated exponential
and logarithmic functions by{\ }$\exp ^{[k]}x=\exp (\exp ^{[k-1]}x)$ and $%
\log ^{[k]}x=\log (\log ^{[k-1]}x)$ with the conventions{\ $\log ^{[0]}x=x$,
$\log ^{[-1]}x=\exp {x}$, $\exp ^{[0]}x=x$ and $\exp ^{[-1]}x=\log x$.}

Using these notions, Biswas et al. introduce the definitions of the $(\alpha
,\beta ,\gamma )$-order and $(\alpha ,\beta ,\gamma )$-type of an analytic
function $f$ in the unit disc $\Delta $ in the following ways:

\begin{definition}
\label{d1.1} (\cite{12})The $(\alpha ,\beta ,\gamma )$-order denoted by $%
\varrho _{(\alpha ,\beta ,\gamma )}[f]$ of a meromorphic function $f$ in $%
\Delta $ is defined by
\begin{equation*}
\varrho _{(\alpha ,\beta ,\gamma )}[f]=\underset{r\longrightarrow 1^{-}}{%
\lim \sup }\frac{\alpha \left( \log T\left( r,f\right) \right) }{\beta
\left( \log \gamma \left( \frac{1}{1-r}\right) \right) }.\newline
\end{equation*}%
If $f$ is an analytic function in $\Delta $, then the $(\alpha ,\beta
,\gamma )$-order is defined by
\begin{equation*}
\varrho _{(\alpha ,\beta ,\gamma ),M}[f]=\underset{r\longrightarrow 1^{-}}{%
\lim \sup }\frac{\alpha \left( \log ^{[2]}M\left( r,f\right) \right) }{\beta
\left( \log \gamma \left( \frac{1}{1-r}\right) \right) }.
\end{equation*}
\end{definition}

Similar to Definition \ref{d1.1}, we can also define the $(\alpha (\log
),\beta ,\gamma )$-order of a meromorphic function $f$ in $\Delta $ in the
following way:

\begin{definition}
\label{d1.2} If $f$ is a meromorphic function in $\Delta $, then
\begin{equation*}
\varrho _{(\alpha (\log ),\beta ,\gamma )}[f]=\underset{r\longrightarrow
1^{-}}{\lim \sup }\frac{\alpha \left( \log ^{[2]}T\left( r,f\right) \right)
}{\beta \left( \log \gamma \left( \frac{1}{1-r}\right) \right) }.
\end{equation*}%
If $f$ is an analytic function in $\Delta $, then
\begin{equation*}
\varrho _{(\alpha (\log ),\beta ,\gamma ),M}[f]=\underset{r\longrightarrow
1^{-}}{\lim \sup }\frac{\alpha \left( \log ^{[3]}M\left( r,f\right) \right)
}{\beta \left( \log \gamma \left( \frac{1}{1-r}\right) \right) }.
\end{equation*}
\end{definition}

\begin{proposition}
\label{p1.1} If $f$ is an analytic function in $\Delta $, then%
\begin{equation*}
\varrho _{(\alpha (\log ),\beta ,\gamma )}[f]=\varrho _{(\alpha (\log
),\beta ,\gamma ),M}[f].
\end{equation*}
\end{proposition}

\begin{proof}
For an analytic function $f$ in $\Delta $, it is well known that (cf. \cite%
{19})%
\begin{equation*}
T(r,f)\leq \log ^{+}M(r,f)\leq \frac{R+r}{R-r}T(R,f)\text{ }(0<r<R<1).
\end{equation*}%
Choosing $R=\frac{1+r}{2}$, we obtain%
\begin{equation}
T\left( r,f\right) \leq \log ^{+}M\left( r,f\right) \leq \frac{1+3r}{1-r}%
T\left( \frac{1+r}{2},f\right) .  \label{1.2}
\end{equation}%
Using the double inequality (\ref{1.2}) and the property
\begin{equation*}
\alpha (a+b)\leq \alpha (a)+\alpha (b)+c
\end{equation*}%
for all $a,b\geq R_{0}$ for some fixed $c\in (0,+\infty ),$ we derive%
\begin{equation*}
\frac{\alpha \left( \log ^{[2]}T\left( r,f\right) \right) }{\beta \left(
\log \gamma \left( \frac{1}{1-r}\right) \right) }\leq \frac{\alpha \left(
\log ^{[3]}M\left( r,f\right) \right) }{\beta \left( \log \gamma \left(
\frac{1}{1-r}\right) \right) }\leq \frac{\alpha \left( \log ^{[2]}\frac{4}{%
1-r}+\log ^{[2]}T\left( \frac{1+r}{2},f\right) +O\left( 1\right) \right) }{%
\beta \left( \log \gamma \left( \frac{1}{1-r}\right) \right) }
\end{equation*}%
\begin{equation*}
\leq \frac{\alpha \left( \log ^{[2]}\frac{1}{1-r}\right) }{\beta \left( \log
\gamma \left( \frac{1}{1-r}\right) \right) }+\frac{\alpha \left( \log
^{[2]}T\left( \frac{1+r}{2},f\right) \right) }{\beta \left( \log \gamma
\left( \frac{1}{1-r}\right) \right) }+\frac{c}{\beta \left( \log \gamma
\left( \frac{1}{1-r}\right) \right) }
\end{equation*}%
\begin{equation}
=\frac{\alpha \left( \log ^{[2]}\frac{1}{1-r}\right) }{\beta \left( \log
\gamma \left( \frac{1}{1-r}\right) \right) }+\frac{\alpha \left( \log
^{[2]}T\left( \frac{1+r}{2},f\right) \right) }{\beta \left( \log \gamma
\left( \frac{1}{1-\frac{1+r}{2}}\right) \right) }\cdot \frac{\beta \left(
\log \gamma \left( \frac{2}{1-r}\right) \right) }{\beta \left( \log \gamma
\left( \frac{1}{1-r}\right) \right) }+\frac{c}{\beta \left( \log \gamma
\left( \frac{1}{1-r}\right) \right) }.  \label{1.3}
\end{equation}%
Since $\gamma \left( \frac{2}{1-r}\right) \leq 2\gamma \left( \frac{1}{1-r}%
\right) $, $\beta \left( x+O\left( 1\right) \right) =\left( 1+o\left(
1\right) \right) \beta \left( x\right) $ and
\begin{equation*}
\alpha \left( \log ^{[2]}x\right) =o\left( \beta (\log \gamma \left(
x\right) \right) \text{ as }x=\frac{1}{1-r}\rightarrow +\infty \text{ when }%
r\rightarrow 1^{-},
\end{equation*}%
letting $r\rightarrow 1^{-}$ and taking the limit superior in (\ref{1.3})
yields the desired equality.
\end{proof}

\begin{definition}
\label{d1.3} (\cite{12}) The $(\alpha ,\beta ,\gamma )$-type denoted by $%
\tau _{(\alpha ,\beta ,\gamma )}[f]$ of a meromorphic function $f$ in $%
\Delta $ with $0<\varrho _{(\alpha ,\beta ,\gamma )}[f]<+\infty $ is defined
by
\begin{equation*}
\tau _{(\alpha ,\beta ,\gamma )}[f]=\underset{r\longrightarrow 1^{-}}{\lim
\sup }\frac{\exp \left( \alpha \left( \log T\left( r,f\right) \right)
\right) }{{\left( \exp \left( \beta \left( \log \gamma \left( \frac{1}{1-r}%
\right) \right) \right) \right) }^{\varrho _{(\alpha ,\beta ,\gamma )}[f]}}.
\end{equation*}%
If $f$ is an analytic function in $\Delta $ with $\varrho _{(\alpha ,\beta
,\gamma )}[f]\in (0,+\infty )$, then the ${(\alpha ,\beta ,\gamma )}$-type
of $f$ is defined by
\begin{equation*}
\tau _{(\alpha ,\beta ,\gamma ),M}[f]=\underset{r\longrightarrow 1^{-}}{\lim
\sup }\frac{\exp {\left( \alpha \left( \log ^{[2]}M\left( r,f\right) \right)
\right) }}{{\left( \exp \left( \beta \left( \log \gamma \left( \frac{1}{1-r}%
\right) \right) \right) \right) }^{\varrho _{(\alpha ,\beta ,\gamma )}[f]}}.
\end{equation*}
\end{definition}

Similar to Definition \ref{d1.3}, we can also define the $(\alpha (\log
),\beta ,\gamma )$-type of a meromorphic function $f$ in $\Delta $ in the
following way:

\begin{definition}
\label{d1.4} The $(\alpha (\log ),\beta ,\gamma )$-type denoted by $\tau
_{(\alpha (\log ),\beta ,\gamma )}[f]$ of a meromorphic function $f$ in $%
\Delta $ with $0<\varrho _{(\alpha (log),\beta ,\gamma )}[f]<+\infty $ is
defined by
\begin{equation*}
\tau _{(\alpha (\log ),\beta ,\gamma )}[f]=\underset{r\longrightarrow 1^{-}}{%
\lim \sup }\frac{\exp {\left( \alpha \left( \log ^{[2]}T\left( r,f\right)
\right) \right) }}{{\left( \exp \left( \beta \left( \log \gamma \left( \frac{%
1}{1-r}\right) \right) \right) \right) }^{\varrho _{(\alpha ,\beta ,\gamma
)}[f]}}.
\end{equation*}%
If $f$ is an analytic function in $\Delta $ with $\varrho _{(\alpha (\log
),\beta ,\gamma )}[f]\in (0,+\infty )$, then the $(\alpha (\log ),\beta
,\gamma )$-type of $f$ is defined by
\end{definition}

\begin{equation*}
\tau _{(\alpha (\log ),\beta ,\gamma ),M}[f]=\underset{r\longrightarrow 1^{-}%
}{\lim \sup }\frac{\exp {\left( \alpha \left( \log ^{[3]}M\left( r,f\right)
\right) \right) }}{{\left( \exp \left( \beta \left( \log \gamma \left( \frac{%
1}{1-r}\right) \right) \right) \right) }^{\varrho _{(\alpha ,\beta ,\gamma
)}[f]}}.
\end{equation*}

The paper is organized as follows. In Section 2, we present our main
results. Section 3 contains some useful lemmas, while the proofs are
presented in Section 4.

\section{Main Results}

In this section, we present the main results of the paper. The following
theorems generalize Theorems \ref{teo1.1}-\ref{teo1.3}, Theorem 2.3 in \cite%
{32} and are counterparts of the results in (\cite{9}), which were proved
for entire functions.

\begin{theorem}
\label{teo2.1} Let $A_{0}(z),A_{1}(z),...,A_{k-1}(z)$ be analytic functions
in $\Delta $ such that
\begin{equation*}
\max \left\{ \varrho _{(\alpha ,\beta ,\gamma
),M}[A_{j}]:j=1,...,k-1\right\} <\varrho _{(\alpha ,\beta ,\gamma
),M}[A_{0}].
\end{equation*}
Then every solution $f(z)\not\equiv 0$ of $\left( \ref{1.1}\right) $
satisfies
\begin{equation*}
\varrho _{(\alpha (\log ),\beta ,\gamma ),M}[f]=\varrho _{(\alpha ,\beta
,\gamma ),M}[A_{0}].
\end{equation*}
\end{theorem}

\begin{theorem}
\label{teo2.2} Let $A_{0}(z),A_{1}(z),...,A_{k-1}(z)$ be analytic functions.
Assume that
\begin{equation*}
\max \{\varrho _{(\alpha ,\beta ,\gamma ),M}[A_{j}]:j=1,...,k-1\}\leq
\varrho _{(\alpha ,\beta ,\gamma ),M}[A_{0}]=\varrho _{0}<+\infty
\end{equation*}%
and
\begin{eqnarray*}
\max \{\tau _{(\alpha ,\beta ,\gamma ),M}[A_{j}] &:&\varrho _{(\alpha ,\beta
,\gamma ),M}[A_{j}]=\varrho _{(\alpha ,\beta ,\gamma ),M}[A_{0}]>0\} \\
&<&\tau _{(\alpha ,\beta ,\gamma ),M}[A_{0}]=\tau _{0}.
\end{eqnarray*}%
Then every solution $f(z)\not\equiv 0$ of $\left( \ref{1.1}\right) $
satisfies $\varrho _{(\alpha (\log ),\beta ,\gamma ),M}[f]=\varrho _{(\alpha
,\beta ,\gamma ),M}[A_{0}].$
\end{theorem}

By combining Theorem \ref{teo2.1} and Theorem \ref{teo2.2}, we obtain the
following result.

\begin{corollary}
\label{c2.1} Let $A_{0}(z),A_{1}(z),...,A_{k-1}(z)$ be analytic functions.
Assume that
\begin{equation*}
\max \{\varrho _{(\alpha ,\beta ,\gamma ),M}[A_{j}]:j=1,...,k-1\}<\varrho
_{(\alpha ,\beta ,\gamma ),M}[A_{0}]=\varrho _{0}<+\infty ,
\end{equation*}%
or%
\begin{eqnarray*}
\max \{\varrho _{(\alpha ,\beta ,\gamma ),M}[A_{j}] &:&j=1,...,k-1\} \\
&\leq &\varrho _{(\alpha ,\beta ,\gamma ),M}[A_{0}]=\varrho _{0}<+\infty
(0<\varrho _{0}<+\infty )
\end{eqnarray*}%
and
\begin{eqnarray*}
\max \{\tau _{(\alpha ,\beta ,\gamma ),M}[A_{j}] &:&\varrho _{(\alpha ,\beta
,\gamma ),M}[A_{j}]=\varrho _{(\alpha ,\beta ,\gamma ),M}[A_{0}]>0\} \\
&<&\tau _{(\alpha ,\beta ,\gamma ),M}[A_{0}]=\tau _{0}\newline
(0<\tau _{0}<+\infty ).
\end{eqnarray*}%
Then every solution $f(z)\not\equiv 0$ of $\left( \ref{1.1}\right) $
satisfies $\varrho _{(\alpha (\log ),\beta ,\gamma ),M}[f]=\varrho _{(\alpha
,\beta ,\gamma ),M}[A_{0}].$
\end{corollary}

\section{Some Auxiliary Lemmas}

In this section, we present several lemmas that will be needed later. The
following lemma is a direct consequence of Theorem 3.1 in (\cite{15}).

\begin{lemma}
\label{lem3.1}(\cite{15}) Let $k$ and $j$ be integers satisfying $k>j\geq 0$%
, and let $\varepsilon >0$ and $d\in (0,1)$. If $f$ is a meromorphic in $%
\Delta $ such that $f^{\left( j\right) }\not\equiv 0$, then
\begin{equation*}
\left\vert \frac{f^{(k)}(z)}{f^{(j)}(z)}\right\vert \leq \left[ \left( \frac{%
1}{1-|z|}\right) ^{2+\varepsilon }\max \left\{ \log \frac{1}{1-|z|}%
;T(s(|z|),f)\right\} \right] ^{k-j}
\end{equation*}%
for $|z|\not\in F,$ where $F\subset \lbrack 0,1)$ is a set of finite
logarithmic measure $m_{l}(F)=\underset{F}{\int }\frac{dr}{1-r}<\infty $,
and where $s(|z|)=1-d(1-r)$. Moreover, if $\varrho \left( f\right) <$ $%
\infty $, then
\begin{equation*}
\left\vert \frac{f^{(k)}(z)}{f^{(j)}(z)}\right\vert \leq \left( \frac{1}{%
1-|z|}\right) ^{\left( k-j\right) \left( \varrho \left( f\right)
+2+\varepsilon \right) }.
\end{equation*}
\end{lemma}

\begin{lemma}
\label{lem3.2}(\cite{19,20}) Let $f$ be a meromorphic function in the unit
disc $\Delta $ and let $k\in \mathbb{N}$. Then
\begin{equation*}
m\left( r,\frac{f^{(k)}}{f}\right) =S(r,f),
\end{equation*}%
where $S(r,f)=O\left( \log ^{+}T(r,f)+\log \left( \frac{1}{1-r}\right)
\right) $, possibly outside a set $F_{1}\subset \lbrack 0,1)$ with finite
logarithmic measure.
\end{lemma}

Here, we present a generalized lemma on the logarithmic derivative in terms
of the $\left( \alpha ,\beta ,\gamma \right) -$order in the unit disc $%
\Delta $.

\begin{lemma}
\label{lem3.3} Let $f$ be a meromorphic function in $\Delta $ of order $%
\varrho _{(\alpha (\log ),\beta ,\gamma )}[f]=\varrho $ and $k\in \mathbb{N}$%
. Then, for any given $\varepsilon >0$,
\begin{equation*}
m\left( r,\frac{f^{(k)}}{f}\right) =O\left( \exp \left\{ \alpha ^{-1}\left(
\left( \varrho +\varepsilon \right) \beta \left( \log \gamma \left( \frac{1}{%
1-r}\right) \right) \right) \right\} \right) ,
\end{equation*}%
outside, possibly, an exceptional set $F_{2}\subset \lbrack 0,1)$ of finite
logarithmic measure.
\end{lemma}

\begin{proof}
Let $k=1$. By the definition of the generalized order
\begin{equation*}
\varrho _{(\alpha (\log ),\beta ,\gamma )}[f]=\underset{r\longrightarrow
1^{-}}{\limsup }\frac{\alpha \left( \log ^{[2]}T\left( r,f\right) \right) }{%
\beta \left( \log \gamma \left( \frac{1}{1-r}\right) \right) }=\varrho ,
\end{equation*}%
it follows that for any $\varepsilon >0$ and all $r$ sufficiently close to $%
1,$%
\begin{equation*}
\alpha \left( \log ^{[2]}T\left( r,f\right) \right) \leq \left( \varrho
+\varepsilon \right) \beta \left( \log \gamma \left( \frac{1}{1-r}\right)
\right) .
\end{equation*}%
Since $\alpha $ is increasing and invertible, we obtain%
\begin{equation*}
\log ^{[2]}T\left( r,f\right) \leq \alpha ^{-1}\left[ \left( \varrho
+\varepsilon \right) \beta \left( \log \gamma \left( \frac{1}{1-r}\right)
\right) \right] ,
\end{equation*}%
and consequently,%
\begin{equation}
T\left( r,f\right) \leq \exp ^{[2]}\left\{ \alpha ^{-1}\left[ \left( \varrho
+\varepsilon \right) \beta \left( \log \gamma \left( \frac{1}{1-r}\right)
\right) \right] \right\} .  \label{3.1}
\end{equation}%
From $\left( \text{\ref{3.1}}\right) $ and Lemma \ref{lem3.2} of the
logarithmic derivative and the assumption
\begin{equation*}
\alpha \left( \log ^{[2]}x\right) =o\left( \beta (\log \gamma \left(
x\right) \right) \text{ as }x=\frac{1}{1-r}\rightarrow +\infty \text{ when }%
r\rightarrow 1^{-},
\end{equation*}%
we deduce that
\begin{align}
m\left( r,\frac{f^{\prime }}{f}\right) & =O\left( \log T\left( r,f\right)
+\log \left( \frac{1}{1-r}\right) \right)  \notag \\
& =O\left( \exp \left\{ \alpha ^{-1}\left( (\varrho +\varepsilon )\beta
\left( \log \gamma \left( \frac{1}{1-r}\right) \right) \right) \right\}
\right) ,\text{ }r\notin F_{2},  \label{3.2}
\end{align}%
where $F_{2}\subset \lbrack 0,1)$ is of finite linear logarithmic measure.%
\newline
Now assume that for some $k\in \mathbb{N}$,
\begin{equation}
m\left( r,\frac{f^{(k)}}{f}\right) =O\left( \exp \left\{ \alpha ^{-1}\left(
(\varrho +\varepsilon )\beta \left( \log \gamma \left( \frac{1}{1-r}\right)
\right) \right) \right\} \right) ,\text{ }r\notin F_{2}.  \label{3.3}
\end{equation}%
Since $N\left( r,f^{(k)}\right) \leq \left( k+1\right) N\left( r,f\right) $,
we have
\begin{align}
T(r,f^{(k)})& =m\left( r,f^{(k)}\right) +N\left( r,f^{(k)}\right) \leq
m\left( r,\frac{f^{(k)}}{f}\right) +m\left( r,f\right) +\left( k+1\right)
N\left( r,f\right)  \notag \\
& \leq \left( k+1\right) T\left( r,f\right) +O\left( \exp \left\{ \alpha
^{-1}\left( \left( \varrho +\varepsilon \right) \beta \left( \log \gamma
\left( \frac{1}{1-r}\right) \right) \right) \right\} \right)  \notag
\end{align}%
and hence, by $\left( \text{\ref{3.1}}\right) $,
\begin{equation}
T(r,f^{(k)})=O\left( \exp ^{[2]}\left\{ \alpha ^{-1}\left( \left( \varrho
+\varepsilon \right) \beta \left( \log \gamma \left( \frac{1}{1-r}\right)
\right) \right) \right\} \right) .  \label{3.4}
\end{equation}

Applying the logarithmic derivative lemma to $f^{(k)}$ and using $\left(
\text{\ref{3.4}}\right) $, we obtain
\begin{align}
m\left( r,\frac{f^{(k+1)}}{f^{(k)}}\right) & =m\left( r,\frac{\left(
f^{(k)}\right) ^{\prime }}{f^{(k)}}\right) =O\left( \log T\left(
r,f^{(k)}\right) +\log \left( \frac{1}{1-r}\right) \right)  \notag \\
& =O\left( \exp \left\{ \alpha ^{-1}\left( \left( \varrho +\varepsilon
\right) \beta \left( \log \gamma \left( \frac{1}{1-r}\right) \right) \right)
\right\} \right) ,\text{ }r\notin F_{2}.  \label{3.5}
\end{align}%
Finally, combining $\left( \text{\ref{3.3}}\right) $ and $\left( \text{\ref%
{3.5}}\right) ,$ we get%
\begin{equation*}
m\left( r,\frac{f^{(k+1)}}{f}\right) \leq m\left( r,\frac{f^{(k+1)}}{f^{(k)}}%
\right) +m\left( r,\frac{f^{(k)}}{f}\right)
\end{equation*}%
\begin{equation*}
=O\left( \exp \left\{ \alpha ^{-1}\left( \left( \varrho +\varepsilon \right)
\beta \left( \log \gamma \left( \frac{1}{1-r}\right) \right) \right)
\right\} \right) ,\text{ }r\notin F_{2}.
\end{equation*}%
This completes the induction and hence the proof.
\end{proof}

To avoid some problems of the exceptional sets, we need the following lemma.

\begin{lemma}
\label{lem3.4} (\cite{1,20}) Let $g:[0,1)\mapsto \mathbb{R}$ and $%
h:[0,1)\mapsto \mathbb{R}$ be monotone non-decreasing functions such that $%
g(r)\leq h(r)$ holds outside of an exceptional set $F_{3}\subset (0,1)$ of
finite logarithmic measure. Then there exists a $d\in (0,1)$ such that if $%
s(r)=1-d(1-r)$, then $g(r)\leq h(s(r))$ for all $r\in \lbrack 0,1)$.
\end{lemma}

\begin{lemma}
\label{lem3.5} Let $f$ be a meromorphic function in $\Delta $. Then $\varrho
_{(\alpha ,\beta ,\gamma )}[f^{\prime }]=\varrho _{(\alpha ,\beta ,\gamma
)}[f]$.
\end{lemma}

\begin{proof}
Let $\varrho _{(\alpha ,\beta ,\gamma )}[f]=\varrho $. By the definition of
the $(\alpha ,\beta ,\gamma )$-order, for any given $\varepsilon >0$ and all
$r$ sufficiently close to $1,$
\begin{equation*}
\alpha \left( \log T\left( r,f\right) \right) \leq \left( \varrho
+\varepsilon \right) \beta \left( \log \gamma \left( \frac{1}{1-r}\right)
\right) ,
\end{equation*}%
which implies%
\begin{equation}
\log T\left( r,f\right) \leq \alpha ^{-1}\left[ \left( \varrho +\varepsilon
\right) \beta \left( \log \gamma \left( \frac{1}{1-r}\right) \right) \right]
.  \label{3.6}
\end{equation}%
Clearly,
\begin{equation*}
T\left( r,f^{\prime }\right) \leq 2T\left( r,f\right) +m\left( r,\frac{%
f^{\prime }}{f}\right) .
\end{equation*}%
By Lemma \ref{lem3.2} on the logarithmic derivative, together with and the
assumption
\begin{equation*}
\alpha \left( \log ^{[2]}x\right) =o\left( \beta (\log \gamma \left(
x\right) \right) \text{ as }x=\frac{1}{1-r}\rightarrow +\infty \text{ when }%
r\rightarrow 1^{-},
\end{equation*}%
we obtain%
\begin{align}
\log T\left( r,f^{\prime }\right) & \leq \log T\left( r,f\right) +\log
\left\{ O\left( \log T\left( r,f\right) +\log \left( \frac{1}{1-r}\right)
\right) \right\} +O(1)  \notag \\
& \leq \log T\left( r,f\right) +\log ^{[2]}T\left( r,f\right) +\log
^{[2]}\left( \frac{1}{1-r}\right) +O(1),\text{ }r\notin F_{4}.  \notag
\end{align}%
Using $\left( \text{\ref{3.6}}\right) $, it follows that%
\begin{equation}
\log T\left( r,f^{\prime }\right) \leq \alpha ^{-1}\left[ \left( \varrho
+4\varepsilon \right) \beta \left( \log \gamma \left( \frac{1}{1-r}\right)
\right) \right] ,\text{ }r\notin F_{4},  \label{3.7}
\end{equation}

where $F_{4}\subset \lbrack 0,1)$ is a set of finite logarithmic measure.

By Lemma \ref{lem3.4}, choosing $d=\frac{1}{2}$ and for all $r\in \lbrack
0,1)$, inequality $\left( \ref{3.7}\right) $ yields
\begin{equation*}
\log T\left( r,f^{\prime }\right) \leq \alpha ^{-1}\left[ \left( \varrho
+4\varepsilon \right) \beta \left( \log \gamma \left( \frac{1}{1-s(r)}%
\right) \right) \right]
\end{equation*}%
\begin{equation*}
\leq \alpha ^{-1}\left[ \left( \varrho +4\varepsilon \right) \beta \left(
\log \gamma \left( \frac{1}{1-\left( 1-\frac{1}{2}\left( 1-r\right) \right) }%
\right) \right) \right]
\end{equation*}%
\begin{equation}
\leq \alpha ^{-1}\left[ \left( \varrho +4\varepsilon \right) \beta \left(
\log \gamma \left( \frac{2}{1-r}\right) \right) \right] .  \label{3.8}
\end{equation}%
Since $\gamma \left( \frac{2}{1-r}\right) \leq 2\gamma \left( \frac{1}{1-r}%
\right) $ and
\begin{equation*}
\beta \left( x+O\left( 1\right) \right) =\left( 1+o\left( 1\right) \right)
\beta \left( x\right) \text{ as }x=\frac{1}{1-r}\rightarrow +\infty \text{
when }r\rightarrow 1^{-},
\end{equation*}%
we obtain
\begin{align*}
\alpha (\log T\left( r,f^{\prime }\right) & \leq \left( \varrho
+4\varepsilon \right) \beta \left( \log \gamma \left( \frac{2}{1-r}\right)
\right) \\
& \leq \left( \varrho +4\varepsilon \right) \beta \left( \log \left( 2\gamma
\left( \frac{1}{1-r}\right) \right) \right) \\
& =\left( \varrho +4\varepsilon \right) \beta \left( \log 2+\log \gamma
\left( \frac{1}{1-r}\right) \right) \\
& =\left( \varrho +4\varepsilon \right) \beta \left( O\left( 1\right) +\log
\gamma \left( \frac{1}{1-r}\right) \right) \\
& =\left( \varrho +4\varepsilon \right) \left( 1+o\left( 1\right) \right)
\beta \left( \log \gamma \left( \frac{1}{1-r}\right) \right) ,
\end{align*}%
Consequently,%
\begin{equation*}
\frac{\alpha (\log T\left( r,f^{\prime }\right) }{\beta \left( \log \gamma
\left( \frac{1}{1-r}\right) \right) }\leq \left( \varrho +4\varepsilon
\right) \left( 1+o\left( 1\right) \right) .
\end{equation*}%
Since $\varepsilon >0$ is arbitrary, we conclude that $\varrho _{(\alpha
,\beta ,\gamma )}[f^{\prime }]\leq \varrho _{(\alpha ,\beta ,\gamma )}[f]$.%
\newline
We now prove the reverse inequality $\varrho _{(\alpha ,\beta ,\gamma
)}[f^{\prime }]\geq \varrho _{(\alpha ,\beta ,\gamma )}[f]$. Let $\varrho
_{(\alpha ,\beta ,\gamma )}[f^{\prime }]=${$\varrho $}$^{\prime }$. By
definition, for any given $\varepsilon >0$ and as $r\rightarrow {1}^{-}$, we
have
\begin{equation}
\log T\left( r,f^{\prime }\right) \leq \exp \left\{ \alpha ^{-1}\left[
\left( {\varrho }^{\prime }+\varepsilon \right) \beta \left( \log \gamma
\left( \frac{1}{1-r}\right) \right) \right] \right\} .  \label{3.9}
\end{equation}%
By a result of C.T. Chuang, see [\cite{14}, Theorem 4.1], see also [\cite{13}%
, p. 281]
\begin{equation}
T\left( r,f\right) \leq O\left( T\left( \frac{r+3}{4},f^{\prime }\right)
+\log \left( \frac{1}{1-r}\right) \right) ,\text{ }r\rightarrow {1}^{-}.
\label{3.10}
\end{equation}%
Using $\gamma \left( \frac{4}{1-r}\right) \leq 4\gamma \left( \frac{1}{1-r}%
\right) $ together with $\beta \left( x+O\left( 1\right) \right) =\left(
1+o\left( 1\right) \right) \beta \left( x\right) $ and the assumption
\begin{equation*}
\alpha \left( \log ^{[2]}x\right) =o\left( \beta (\log \gamma \left(
x\right) \right) \text{ as }x=\frac{1}{1-r}\rightarrow +\infty \text{ when }%
r\rightarrow 1^{-},
\end{equation*}%
we deduce from $\left( \ref{3.9}\right) $ and $\left( \ref{3.10}\right) $
that
\begin{align*}
T\left( r,f\right) & \leq O\left( \exp \left\{ \alpha ^{-1}\left[ \left( {%
\varrho }^{\prime }+\varepsilon \right) )\beta \left( \log \gamma \left(
\frac{4}{1-r}\right) \right) \right] \right\} +\log \left( \frac{1}{1-r}%
\right) \right) \\
& \leq \exp \left\{ \alpha ^{-1}\left[ \left( {\varrho }^{\prime
}+2\varepsilon \right) \beta \left( \log \left( \gamma \left( \frac{4}{1-r}%
\right) \right) \right) \right] \right\} \\
& \leq \exp \left\{ \alpha ^{-1}\left[ \left( {\varrho }^{\prime
}+2\varepsilon \right) \left( 1+o\left( 1\right) \right) \beta \left( \log
\gamma \left( \frac{1}{1-r}\right) \right) \right] \right\} .
\end{align*}%
Consequently,%
\begin{equation*}
\log T\left( r,f\right) \leq \alpha ^{-1}\left[ \left( {\varrho }^{\prime
}+2\varepsilon \right) \left( 1+o\left( 1\right) \right) \beta \left( \log
\gamma \left( \frac{1}{1-r}\right) \right) \right] ,
\end{equation*}%
and hence%
\begin{equation*}
\alpha \left( \log T\left( r,f\right) \right) \leq \left( {\varrho }^{\prime
}+2\varepsilon \right) \left( 1+o\left( 1\right) \right) \beta \left( \log
\gamma \left( \frac{1}{1-r}\right) \right) ,\text{ }r\rightarrow {1}^{-}.
\end{equation*}%
Since, $\varepsilon >0$ is arbitrary, this yields $\varrho _{(\alpha ,\beta
,\gamma )}[f^{\prime }]\geq \varrho _{(\alpha ,\beta ,\gamma )}[f].$
Combining both inequalities, we conclude that $\varrho _{(\alpha ,\beta
,\gamma )}[f^{\prime }]=\varrho _{(\alpha ,\beta ,\gamma )}[f]$ which
completes the proof.
\end{proof}

\begin{remark}
\label{r3.1} In accordance with Lemma \ref{lem3.5}, one readily deduces that
$\varrho _{(\alpha (\log ),\beta ,\gamma )}[f^{\prime }]=\varrho _{(\alpha
(\log ),\beta ,\gamma )}[f]$, where $f$ is a meromorphic function in $\Delta
$.
\end{remark}

\begin{lemma}
\label{lem3.6} Let $f$ be an analytic function in the unit disc $\Delta $
such that $0<\varrho _{(\alpha ,\beta ,\gamma ),M}[f]=\varrho _{0}<+\infty $%
. Then, for any $0<\mu <\varrho _{0}$, there exists a set $I\in \lbrack 0,1)$
of infinite logarithmic measure $m_{l}(I)=\int\limits_{I}\frac{dr}{1-r}%
=+\infty ,$ such that for all $r\in I$ one has
\begin{equation*}
\alpha \left( \log ^{[2]}M\left( r,f\right) \right) >\mu \beta \left( \log
\gamma \left( \frac{1}{1-r}\right) \right) .
\end{equation*}
\end{lemma}

\begin{proof}
By the definition of the upper limit, there exists an increasing sequence $%
\{r_{m}\}$\ with $r_{m}\rightarrow 1^{-}$ as $m\rightarrow \infty $ such
that
\begin{equation*}
1-\left( 1-\frac{1}{m}\right) (1-r_{m})<r_{m+1}\text{ }
\end{equation*}%
and
\begin{equation*}
\underset{m\rightarrow +\infty }{\lim }~\frac{\alpha \left( \log
^{[2]}M\left( r_{m},f\right) \right) }{\beta \left( \log \gamma \left( \frac{%
1}{1-r_{m}}\right) \right) }=\varrho _{0}.
\end{equation*}%
Hence, there exists an integer $m_{0}$ such that for all $m\geq m_{0}$ and
for any given $\varepsilon $ satisfying$~0<\varepsilon <\varrho _{0}-\mu $
\begin{equation}
\alpha \left( \log ^{[2]}M\left( r_{m},f\right) \right) >\left( \varrho
_{0}-\varepsilon \right) \beta \left( \log \gamma \left( \frac{1}{1-r_{m}}%
\right) \right) .  \label{3.11}
\end{equation}%
For $r\in \left[ r_{m},1-\left( 1-\frac{1}{m}\right) \left( 1-r_{m}\right) %
\right] ,$ we observe that%
\begin{equation*}
\underset{m\rightarrow +\infty }{\lim }\frac{\beta \left( \log \gamma \left(
\left( 1-\frac{1}{m}\right) \frac{1}{1-r}\right) \right) }{\beta \left( \log
\gamma \left( \frac{1}{1-r}\right) \right) }=1.
\end{equation*}%
Therefore, for a given $\mu $ with $0<\mu <\varrho _{0}-\varepsilon $, there
exists an integer $m_{1}$ such that for $m\geq m_{1}$ we have
\begin{equation}
\frac{\beta \left( \log \gamma \left( \left( 1-\frac{1}{m}\right) \frac{1}{%
1-r}\right) \right) }{\beta \left( \log \gamma \left( \frac{1}{1-r}\right)
\right) }>\frac{\mu }{\varrho _{0}-\varepsilon }.  \label{3.12}
\end{equation}%
Combining $\left( \ref{3.11}\right) $ and $\left( \ref{3.12}\right) $ for
all $m\geq m_{2}:=\max \left\{ {m_{0},m_{1}}\right\} $ and for any
\begin{equation*}
r\in \left[ r_{m},1-\left( 1-\frac{1}{m}\right) \left( 1-r_{m}\right) \right]
,
\end{equation*}%
we obtain
\begin{align*}
\alpha \left( \log ^{[2]}M\left( r,f\right) \right) & \geq \alpha \left(
\log ^{[2]}M\left( r_{m},f\right) \right) >\left( \varrho _{0}-\varepsilon
\right) \beta \left( \log \gamma \left( \frac{1}{1-r_{m}}\right) \right)  \\
& >\left( \varrho _{0}-\varepsilon \right) \beta \left( \log \gamma \left(
\left( 1-\frac{1}{m}\right) \frac{1}{1-r}\right) \right)  \\
& >\left( \varrho _{0}-\varepsilon \right) \frac{\mu }{\varrho
_{0}-\varepsilon }\beta \left( \log \gamma \left( \frac{1}{1-r}\right)
\right) =\mu \beta \left( \log \gamma \left( \frac{1}{1-r}\right) \right) .
\end{align*}%
Finally, let $I=\underset{m=m_{2}}{\overset{+\infty }{\cup }}I_{m}$ where $%
I_{m}=\left[ r_{m},1-\left( 1-\frac{1}{m}\right) \left( 1-r_{m}\right) %
\right] $. Then,
\begin{equation*}
m_{l}(I)=\sum_{m=m_{2}}^{+\infty }\int\limits_{I_{m}}\frac{dr}{1-r}%
=\sum_{m=m_{2}}^{+\infty }\log \left( \frac{m}{m-1}\right) =+\infty .
\end{equation*}%
This completes the proof.
\end{proof}

We can also prove the following result by using similar reason as in the
proof of Lemma \ref{lem3.6}.

\begin{lemma}
\label{lem3.7} let $f$ be an analytic function in $\Delta $ with $\varrho
_{(\alpha ,\beta ,\gamma ),M}[f]=\varrho _{0}\in (0,+\infty )$ and $\tau
_{(\alpha ,\beta ,\gamma ),M}[f]\in (0,+\infty )$. Then for any given $%
\omega <\tau _{(\alpha ,\beta ,\gamma ),M}[f]$, there exists a set $%
I_{1}\subset (0,1)$ of infinite logarithmic measure such that for all $r\in
I_{1}$
\begin{equation*}
\exp \left\{ \alpha \left( \log ^{[2]}\left( M(r,f)\right) \right) \right\}
>\omega \left( \exp \left\{ \beta \left( \log \left( \gamma \left( \frac{1}{%
1-r}\right) \right) \right) \right\} \right) ^{\varrho _{0}}.
\end{equation*}
\end{lemma}

\begin{lemma}
\label{lem3.8} (\cite{21}) Let $f$ be a solution of equation $\left( \ref%
{1.1}\right) $, where the coefficients $A_{j}\left( z\right) $ $%
(j=0,...,k-1) $ are analytic functions on the disc $\Delta _{R}=\left\{ z\in
\mathbb{C}:|z|<R\right\} ,0<R\leq \infty $. Let $n_{c}\in \{1,...,k\}$ be
the number of nonzero coefficients $A_{j}\left( z\right) $ $(j=0,...,k-1)$,
and let $\theta \in \lbrack 0,2\pi ]$ and $\varepsilon >0$. If $z_{\theta
}=\nu e^{i\theta }\in \Delta _{R}$ is such that $A_{j}\left( z_{\theta
}\right) \neq 0$ for some $j=0,...,k-1$, then for all $\nu <r<R$,
\begin{equation*}
\left\vert f(re^{i\theta })\right\vert \leq C\exp \left(
n_{c}\int\limits_{\nu }^{r}\max_{j=0,...,k-1}\left\vert A_{j}(te^{i\theta
})\right\vert ^{\frac{1}{k-j}}dt\right) ,
\end{equation*}%
where $C>0$ is a constant satisfying
\begin{equation*}
C\leq (1+\varepsilon )\max_{j=0,...,k-1}\left( \frac{|f^{(j)}(z_{(\theta )}|%
}{(n_{c})^{j}\max_{j=0,...,k-1}|A_{j}(z_{\theta })|^{\frac{j}{k-n}}}\right) .
\end{equation*}
\end{lemma}

\begin{lemma}
\label{lem3.9} Let $A_{0}(z),A_{1}(z),...,A_{k-1}(z)$ be analytic functions
in the disc $\Delta $. Then every non-trivial solution $f$ of $\left( \ref%
{1.1}\right) $ satisfies
\begin{equation*}
\varrho _{(\alpha (\log ),\beta ,\gamma ),M}[f]\leq \max \left\{ \varrho
_{(\alpha ,\beta ,\gamma ),M}[A_{j}],j=0,...,k-1\right\} .
\end{equation*}
\end{lemma}

\begin{proof}
Set $\varrho =\max \left\{ \varrho _{(\alpha ,\beta ,\gamma
),M}[A_{j}],j=0,...,k-1\right\} $. Let $f\not\equiv 0$ be a solution of $%
\left( \ref{1.1}\right) $. Let $\theta _{0}\in \lbrack 0,2\pi ]$ be such
that $|f(re^{i\theta })|=M(r,f)$ \newline
By Lemma \ref{lem3.8}, we have
\begin{align}
M(r,f)& \leq C\exp \left( n_{c}\int\limits_{\nu
}^{r}\max_{j=0,...,k-1}\left\vert A_{j}(te^{i\theta })\right\vert ^{\frac{1}{%
k-j}}dt\right)  \notag \\
& \leq C\exp \left( n_{c}\int\limits_{\nu }^{r}\max_{j=0,...,k-1}\left(
M(r,A_{j})\right) ^{\frac{1}{k-j}}dt\right)  \notag \\
& \leq C\exp \left( n_{c}\left( r-\nu \right)
\max_{j=0,...,k-1}M(r,A_{j})\right) .  \label{3.13}
\end{align}%
By the definition of $\varrho _{(\alpha ,\beta ,\gamma ),M}[A_{j}]$,
\begin{equation}
M(r,A_{j})\leq \exp ^{[2]}\left\{ \alpha ^{-1}\left( \left( \varrho +\frac{%
\varepsilon }{2}\right) \beta \left( \log \gamma \left( \frac{1}{1-r}\right)
\right) \right) \right\} ,\text{ }j=0,...,k-1  \label{3.14}
\end{equation}%
holds for any $\varepsilon >0$. Hence from $\left( \ref{3.12}\right) $ and $%
\left( \ref{3.13}\right) $ we obtain
\begin{equation*}
\varrho _{(\alpha (\log ),\beta ,\gamma ),M}[f]\leq \varrho +\varepsilon .
\end{equation*}%
Since $\varepsilon >0$ is arbitrary, we have
\begin{equation*}
\varrho _{(\alpha (\log ),\beta ,\gamma ),M}[f]\leq \max \left\{ \varrho
_{(\alpha ,\beta ,\gamma ),M}[A_{j}],j=0,...,k-1\right\}
\end{equation*}
\end{proof}

\section{Proof of the Main Results}

\textbf{Proof of Theorem \ref{teo2.1}.}

\begin{proof}
Set $\varrho _{j}=\varrho _{(\alpha ,\beta ,\gamma ),M}[A_{j}]$ $\left(
j=0,...,k-1\right) $ and $\varrho _{f}=\varrho _{(\alpha (\log ),\beta
,\gamma ),M}[f].$\newline
First, we prove that $\varrho _{f}\geq \varrho _{0}$. Suppose the contrary $%
\varrho _{f}<\varrho _{0}$. Let $f\not\equiv 0$ be a solution of the
equation $\left( \ref{1.1}\right) $. In accordance with $\left( \ref{1.1}%
\right) $ we have
\begin{equation}
|A_{0}(z)|\leq \left\vert \frac{f^{(k)}(z)}{f(z)}\right\vert
+|A_{k-1}(z)|\left\vert \frac{f^{(k-1)}(z)}{f(z)}\right\vert +\cdots
+|A_{1}(z)|\left\vert \frac{f^{\prime }(z)}{f(z)}\right\vert .  \label{4.1}
\end{equation}%
Since the functions $A_{j}$ are analytic in $\Delta $ and satisfy $\varrho
_{j}<\varrho _{0}~\left( j=1,...,k-1\right) $, then there exists a constant $%
\lambda _{1}>0$ such that $\varrho _{j}<\lambda _{1}<\varrho _{0}~~\left(
j=1,...,k-1\right) $. Hence, as $~r\rightarrow 1^{-}$
\begin{equation}
M(r,A_{j})<\exp ^{[2]}\left\{ \alpha ^{-1}\left( \lambda _{1}\left( \beta
\left( \log \gamma \left( \frac{1}{1-r}\right) \right) \right) \right)
\right\} ~~~~~  \label{4.2}
\end{equation}%
Without loss of generality, we may assume that
\begin{equation}
\varrho _{f}<\lambda _{1}<\varrho _{0}.  \label{4.3}
\end{equation}%
Applying Lemma \ref{lem3.6} to the coefficient $A_{0}(z)$ with a constant $%
\lambda _{2}$ satisfying $\lambda _{1}<\lambda _{2}<\varrho _{0}$, we obtain
\begin{equation}
M(r,A_{0})>\exp ^{[2]}\left\{ \alpha ^{-1}\left( \lambda _{2}\left( \beta
\left( \log \gamma \left( \frac{1}{1-r}\right) \right) \right) \right)
\right\} ~~~~r\in I,\text{ \ }r\rightarrow 1^{-},  \label{4.4}
\end{equation}%
where $I\subset \lbrack 0,1)$ is a set of infinite logarithmic measure. The
Lemma \ref{lem3.1} implies the following estimate for $j=1,...,k$
\begin{equation}
\left\vert \frac{f^{(j)}(z)}{f(z)}\right\vert \leq \left( \left( \frac{1}{%
1-|z|}\right) ^{2+2\varepsilon }T\left( s(|z|),f\right) \right)
^{j},~~~|z|\notin F,  \label{4.5}
\end{equation}%
where $F\subset \lbrack 0,1)$ is a set of finite logarithmic measure.\newline
Since $I\setminus F$ has infinite logarithmic measure, there exists a
sequence $\left\{ z_{n}\right\} $ with $|z_{n}|=r_{n}\in I\setminus F$ such
that $r_{n}\rightarrow 1^{-}$. Let $s(|z_{n}|)=R_{n}$. Then $1-|z_{n}|=\frac{%
1}{d}(1-R_{n}),$ $d\in (0,1)$.\newline
Using $\left( \ref{4.2}\right) $, $\left( \ref{4.4}\right) $, $\left( \ref%
{4.5}\right) $ together with assumption $\left( \ref{4.3}\right) $, we
derive from $\left( \ref{4.1}\right) $ that
\begin{align*}
& \exp ^{[2]}\left\{ \alpha ^{-1}\left( \lambda _{2}\left( \beta \left( \log
\gamma \left( \frac{d}{1-R_{n}}\right) \right) \right) \right) \right\} \leq
\left( \left( \frac{d}{1-R_{n}}\right) ^{2+2\varepsilon }T\left(
R_{n},f\right) \right) ^{k}+ \\
& +\left( \left( \left( \frac{d}{1-R_{n}}\right) ^{2+2\varepsilon }T\left(
R_{n},f\right) \right) ^{k-1}+\cdots +\left( \frac{d}{1-R_{n}}\right)
^{2+2\varepsilon }T\left( R_{n},f\right) \right) \times \\
& \times \exp ^{[2]}\left\{ \alpha ^{-1}\left( \lambda _{1}\left( \beta
\left( \log \gamma \left( \frac{d}{1-R_{n}}\right) \right) \right) \right)
\right\} \leq \\
& \leq k\left( \left( \frac{d}{1-R_{n}}\right) ^{2+2\varepsilon }T\left(
R_{n},f\right) \right) ^{k}\exp ^{[2]}\left\{ \alpha ^{-1}\left( \lambda
_{1}\left( \beta \left( \log \gamma \left( \frac{d}{1-R_{n}}\right) \right)
\right) \right) \right\} \\
& \leq k\left( \left( \frac{d}{1-R_{n}}\right) ^{2+2\varepsilon }\exp
^{[2]}\left\{ \alpha ^{-1}\left( \lambda _{1}\left( \beta \left( \log \gamma
\left( \frac{1}{1-R_{n}}\right) \right) \right) \right) \right\} \right)
^{k}\times \\
& \times \exp ^{[2]}\left\{ \alpha ^{-1}\left( \lambda _{1}\left( \beta
\left( \log \gamma \left( \frac{d}{1-R_{n}}\right) \right) \right) \right)
\right\} \\
& \leq \left( \exp ^{[2]}\left\{ \alpha ^{-1}\left( (\lambda
_{1}+\varepsilon )\left( \beta \left( \log \gamma \left( \frac{d}{1-R_{n}}%
\right) \right) \right) \right) \right\} \right) ^{k+2} \\
& \leq \exp ^{[2]}\left\{ \alpha ^{-1}\left( (\lambda _{1}+2\varepsilon
)\left( \beta \left( \log \left( \gamma \left( \frac{d}{1-R_{n}}\right)
\right) \right) \right) \right) \right\} ,~~R_{n}\in I\setminus
F,~R_{n}\rightarrow 1^{-}.
\end{align*}%
By arbitrariness of $\varepsilon >0$ and the monotony of the function $%
\alpha ^{-1}$, we obtain that $\lambda _{1}\geq \lambda _{2}$. This
contradiction proves the inequality $\varrho _{0}\leq \varrho _{f}$.\newline
Second, we prove that $\varrho _{0}\geq \varrho _{f}$. By using Lemma \ref%
{lem3.9}, we obtain
\begin{equation*}
\varrho _{f}=\varrho _{(\alpha (\log ),\beta ,\gamma ),M}[f]\leq \max
\left\{ \varrho _{(\alpha ,\beta ,\gamma ),M}[A_{j}],j=0,...,k-1\right\}
=\varrho _{(\alpha ,\beta ,\gamma ),M}[A_{0}]=\varrho _{0}.
\end{equation*}%
From the last inequality and $\varrho _{0}\leq \varrho _{f}$ we obtain $%
\varrho _{f}=\varrho _{0}$.\newline
\end{proof}

\textbf{Proof of Theorem \ref{teo2.2}}{\ }

\begin{proof}
Suppose that $f$ is a nontrivial solution of $\left( \ref{1.1}\right) $.
From $\left( \ref{1.1}\right) $, we can write that
\begin{equation}
|A_{0}(z)|\leq \left\vert \frac{f^{(k)}(z)}{f(z)}\right\vert
+|A_{k-1}(z)|\left\vert \frac{f^{(k-1)}(z)}{f(z)}\right\vert +\cdots
+|A_{1}(z)|\left\vert \frac{f^{\prime }(z)}{f(z)}\right\vert  \label{4.6}
\end{equation}%
If%
\begin{equation*}
\max \{\varrho _{(\alpha ,\beta ,\gamma ),M}[A_{j}],j=1,...,k-1\}<\varrho
_{(\alpha ,\beta ,\gamma ),M}[A_{0}]=\varrho _{0}<+\infty ,
\end{equation*}%
then by Theorem \ref{teo2.1}, we obtain that
\begin{equation*}
\varrho _{(\alpha (\log ),\beta ,\gamma ),M}[f]=\varrho _{(\alpha ,\beta
,\gamma ),M}[A_{0}].
\end{equation*}%
Suppose that
\begin{equation*}
\max \left\{ \varrho _{(\alpha ,\beta ,\gamma
),M}[A_{j}],j=1,...,k-1\right\} =\varrho _{(\alpha ,\beta ,\gamma
),M}[A_{0}]=\varrho _{0}<+\infty
\end{equation*}%
and
\begin{equation*}
\max \{\tau _{(\alpha ,\beta ,\gamma ),M}[A_{j}]:\varrho _{(\alpha ,\beta
,\gamma ),M}[A_{j}]=\varrho _{(\alpha ,\beta ,\gamma ),M}[A_{0}]>0\}<\tau
_{(\alpha ,\beta ,\gamma ),M}[A_{0}]=\tau _{M}<+\infty .
\end{equation*}%
First, we prove that
\begin{equation*}
\varrho _{f}=\varrho _{(\alpha (\log ),\beta ,\gamma ),M}[f]\geq \varrho
_{(\alpha ,\beta ,\gamma ),M}[A_{0}]=\varrho _{0}
\end{equation*}%
By assumptions, there exists a set $K\subseteq \{1,2,...,k-1\}$ such that
\begin{equation*}
\varrho _{(\alpha ,\beta ,\gamma ),M}[A_{j}]=\varrho _{(\alpha ,\beta
,\gamma ),M}[A_{0}]=\varrho _{0},\text{ }\tau _{(\alpha ,\beta ,\gamma
),M}[A_{j}]<\tau _{(\alpha ,\beta ,\gamma ),M}[A_{0}],\text{ }j\in K
\end{equation*}%
and
\begin{equation*}
\varrho _{(\alpha ,\beta ,\gamma ),M}[A_{j}]<\varrho _{(\alpha ,\beta
,\gamma ),M}[A_{0}]=\varrho _{0},\text{ }j\in \{1,2,...,k-1\}\setminus K
\end{equation*}%
Thus, we choose $\lambda _{3}$ and $\lambda _{4}$ satisfying
\begin{equation*}
\max \{\tau _{(\alpha ,\beta ,\gamma ),M}[A_{j}]:j\in K\}<\lambda
_{3}<\lambda _{4}<\tau _{(\alpha ,\beta ,\gamma ),M}[A_{0}]=\tau _{M}.
\end{equation*}%
For $r\rightarrow 1^{-},$ we have
\begin{equation}
|A_{j}(z)|\leq \exp ^{[2]}\left\{ \alpha ^{-1}\left( \log \left( \lambda
_{3}\left( \exp \left( \beta \left( \log \gamma \left( \frac{1}{1-r}\right)
\right) \right) \right) ^{\varrho _{0}}\right) \right) \right\} ,~j\in K
\label{4.7}
\end{equation}%
and%
\begin{align}
& |A_{j}(z)|\leq \exp ^{[2]}\left\{ \alpha ^{-1}\left( \log \left( \exp
\left( \beta \left( \log \gamma \left( \frac{1}{1-r}\right) \right) \right)
\right) ^{\xi }\right) \right\}  \notag \\
& \leq \exp ^{[2]}\left\{ \alpha ^{-1}\left( \log \left( \lambda _{3}\left(
\exp \left( \beta \left( \log \gamma \left( \frac{1}{1-r}\right) \right)
\right) \right) ^{\varrho _{0}}\right) \right) \right\} ~,\text{ }j\in
\{1,2,...,k-1\}\setminus K,  \label{4.8}
\end{align}%
where $0<\xi <\varrho _{0}$. By Lemma \ref{3.7}, there exists a set $%
I_{1}\subset (1,+\infty )$ with infinite logarithmic measure, such that for
all $r\in I_{1}$, there holds
\begin{equation}
|A_{0}(z)|>\exp ^{[2]}\left\{ \alpha ^{-1}\left( \log \left( \lambda
_{4}\left( \exp \left( \beta \left( \log \gamma \left( \frac{1}{1-r}\right)
\right) \right) \right) ^{\varrho _{0}}\right) \right) \right\} .
\label{4.9}
\end{equation}%
By $\left( \ref{4.5}\right) ,$ the following estimate holds for $j=1,...,k$
\begin{equation}
\left\vert \frac{f^{(j)}(z)}{f(z)}\right\vert \leq \left( \left( \frac{1}{%
1-|z|}\right) ^{2+2\varepsilon }T\left( s(|z|),f\right) \right)
^{j},~~~|z|\notin F,  \label{4.10}
\end{equation}%
where $F\subset \lbrack 0,1)$ is a set of finite logarithmic measure. By $%
\left( \ref{4.10}\right) $ and the Definition \ref{d1.2} of $(\alpha (\log
),\beta ,\gamma )$-order of $f,$ we get%
\begin{equation*}
\left\vert \frac{f^{(j)}(z)}{f(z)}\right\vert \leq \left( \left( \frac{1}{%
1-|z|}\right) ^{2+2\varepsilon }T\left( s(|z|),f\right) \right) ^{j}
\end{equation*}%
\begin{equation}
\leq \left( \exp ^{[2]}\left\{ \alpha ^{-1}\left( (\varrho _{f}+\varepsilon
)\beta \left( \log \left( \gamma \left( \frac{1}{1-s(|z|)}\right) \right)
\right) \right) \right\} \right) ^{k+1},~j=1,...,k.  \label{4.11}
\end{equation}%
Since $I_{1}\setminus F$ is of infinite logarithmic measure, there exists a
sequence $\left\{ z_{n}\right\} $ with $|z_{n}|=r_{n}\in I_{1}\setminus F$
such that $r_{n}\rightarrow 1^{-}$. Setting $s(|z_{n}|)=R_{n}$, we obtain $%
1-|z_{n}|=\frac{1}{d}(1-R_{n}),$ $d\in (0,1)$.\newline
Therefore, by substituting $\left( \ref{4.7}\right) $, $\left( \ref{4.8}%
\right) $, $\left( \ref{4.9}\right) $ and $\left( \ref{4.11}\right) $ into $%
\left( \ref{4.6}\right) $, we obtain, for $R_{n}\in I_{1}\setminus F$ with $%
R_{n}\rightarrow 1^{-}$ that
\begin{equation*}
\exp ^{[2]}\left\{ \alpha ^{-1}\left( \log \left( \lambda _{4}\left( \exp
\left( \beta \left( \log \gamma \left( \frac{d}{1-R_{n}}\right) \right)
\right) \right) ^{\varrho _{0}}\right) \right) \right\}
\end{equation*}%
\begin{align}
& \leq k\exp ^{[2]}\left\{ \alpha ^{-1}\left( \log \left( \lambda _{3}\left(
\exp \left( \beta \left( \log \gamma \left( \frac{d}{1-R_{n}}\right) \right)
\right) \right) ^{\varrho _{0}}\right) \right) \right\}  \notag  \label{11}
\\
& \times \left[ \exp ^{[2]}\left\{ \alpha ^{-1}\left( (\varrho
_{f}+\varepsilon )\beta \left( \log \gamma \left( \frac{1}{1-R_{n}}\right)
\right) \right) \right\} \right] ^{k+1}  \notag
\end{align}%
\begin{equation}
\leq \exp ^{[2]}\left\{ \alpha ^{-1}\left( \log \left( (\lambda
_{3}+2\varepsilon )\left( \exp \left( \beta \left( \log \gamma \left( \frac{d%
}{1-R_{n}}\right) \right) \right) \right) ^{\varrho _{0}}\right) \right)
\right\} .  \label{4.12}
\end{equation}%
From $\left( \ref{4.12}\right) ,$ we get that $\lambda _{3}\geq \lambda _{4}$%
. This contradiction implies
\begin{equation*}
\varrho _{(\alpha (\log ),\beta ,\gamma ),M}[f]\geq \varrho _{(\alpha ,\beta
,\gamma ),M}[A_{0}].
\end{equation*}%
On the other hand, by Lemma \ref{lem3.9}, we get that
\begin{equation*}
\varrho _{(\alpha (\log ),\beta ,\gamma ),M}[f]\leq \max \left\{ \varrho
_{(\alpha ,\beta ,\gamma ),M}[A_{j}],j=0,...,k-1\right\} =\varrho _{(\alpha
,\beta ,\gamma ),M}[A_{0}].
\end{equation*}%
Hence every nontrivial solution $f$ of $\left( \ref{1.1}\right) $ satisfies $%
\varrho _{(\alpha (\log ),\beta ,\gamma ),M}[f]=\varrho _{(\alpha ,\beta
,\gamma ),M}[A_{0}].$
\end{proof}

\end{document}